\documentclass[11pt,reqno]{amsart}

\usepackage{amssymb}
\usepackage{amsmath}
\usepackage{amsthm}
\usepackage{mathrsfs}
\usepackage{dsfont}
\usepackage{enumitem}
\usepackage[left=2.7cm,right=2.7cm,top=2.2cm,bottom=2.2cm]{geometry}
\usepackage{mathtools}
\usepackage{hyperref}
\usepackage{aliascnt}
\usepackage[capitalise,noabbrev,nameinlink]{cleveref}
\hypersetup{
	colorlinks=true,
	linkcolor=red
}
\usepackage{xcolor}

\numberwithin{equation}{section}

\newaliascnt{Theorem}{equation}
\newtheorem{Theorem}[Theorem]{Theorem}
\aliascntresetthe{Theorem}
\newaliascnt{Lemma}{equation}
\newtheorem{Lemma}[Lemma]{Lemma}
\aliascntresetthe{Lemma}
\newaliascnt{Corollary}{equation}
\newtheorem{Corollary}[Corollary]{Corollary}
\aliascntresetthe{Corollary}
\newaliascnt{Proposition}{equation}
\newtheorem{Proposition}[Proposition]{Proposition}
\aliascntresetthe{Proposition}
\newaliascnt{Claim}{equation}

\aliascntresetthe{Claim}
\newaliascnt{Question}{equation}

\aliascntresetthe{Question}
\newaliascnt{Conjecture}{equation}

\aliascntresetthe{Conjecture}
\newaliascnt{Remark}{equation}
\newtheorem{Remark}[Remark]{Remark}
\aliascntresetthe{Remark}
\theoremstyle{definition}
\newaliascnt{Definition}{equation}

\aliascntresetthe{Definition}

\crefformat{equation}{(#2#1#3)}
\Crefformat{equation}{(#2#1#3)}
\crefrangeformat{equation}{(#3#1#4)--(#5#2#6)}
\Crefrangeformat{equation}{(#3#1#4)--(#5#2#6)}
\crefmultiformat{equation}{(#2#1#3)}{ and~(#2#1#3)}{, (#2#1#3)}{ and~(#2#1#3)}
\Crefmultiformat{equation}{(#2#1#3)}{ and~(#2#1#3)}{, (#2#1#3)}{ and~(#2#1#3)}

\setlist{nosep}

\DeclareMathOperator{\Tr}{Tr}

\title{Backward Pseudolocality of \texorpdfstring{$G_2$}{G2}-Laplacian Flow}

\author{Weiyong He}
\address{Department of Mathematics, University of Oregon,
Eugene, OR 97403, USA}
\email{whe@uoregon.edu}

\author{Longhuan Yan}
\address{School of Mathematical Sciences, University of Science and Technology of China, No.96, JinZhai Road, Hefei, Anhui}
\email{yanlonghuan@mail.ustc.edu.cn}

\date{\today}

\begin{document}

\begin{abstract}
	In this paper, we establish a new estimate on the time-derivative of pointed Nash entropy and use it to give a new proof of backward pseudolocality for Ricci flow.
	The proof does not rely on Bamler's compactness theorem and structure theory concerning limits of Ricci flows.
	Our methods also work for other reasonable geometric flows, like the $G_2$-Laplacian, at a scale depending on its deviation from Ricci flow and its scalar curvature.
\end{abstract}

\maketitle
\tableofcontents
\bibliographystyle{amsalpha}

\newcommand{\C}{{\mathbb C}}
\newcommand{\Q}{{\mathbb Q}}
\newcommand{\R}{{\mathbb R}}
\newcommand{\Z}{{\mathbb Z}}
\newcommand{\N}{{\mathbb N}}
\newcommand*\norm[1]{\left\lVert#1\right\rVert}
\newcommand*\normm[1]{\lVert#1\rVert}
\newcommand*\abs[1]{\left\lvert#1\right\rvert}
\newcommand*\abss[1]{\lvert#1\rvert}
\newcommand*\inn[2]{\langle#1,#2\rangle}
\renewcommand{\geq}{\geqslant}
\renewcommand{\leq}{\leqslant}

\newcommand{\Ric}{{\mathrm{Ric}}}
\newcommand{\Rm}{{\mathrm{Rm}}}

\section{Introduction}
The primary goal of this paper is to study the $G_2$-Laplacian flow with bounded torsion.
We refer to the foundational paper \cite{lotay2017laplacian} of Lotay--Wei for background and references concerning $G_2$-Laplacian flow.
One of their main results is that on a compact manifold the flow extends once the Ricci curvature is uniformly bounded, and they raised the question whether the flow extends once the torsion, or equivalently the scalar curvature, remains uniformly bounded.
The first step toward this is given by Chen in \cite{chen2018shi}, where he proved a no local collapsing theorem by using Perelman's entropy functional.
To some extent, \cite{chen2018shi} shed light on the possibility of carrying over the entropy approach for Ricci flow to our setting of Laplacian flow with bounded torsion.
Several later developments in this direction are \cite{ishida2024logarithmic,he2026pseudolocality}.
In this paper, we will pursue this idea and prove backward pseudolocality for $G_2$-Laplacian flow with bounded torsion.

\subsection{Backward pseudolocality}
Perelman's pseudolocality \cite[\S 10]{perelman2002entropy} provides a quantitative substitute for strict locality in geometric heat-type equations: sufficiently regular initial local geometry ensures short-time local curvature control.
It is fundamental to the study of Ricci flow.
In applications, one often needs an equally important companion notion called backward pseudolocality, which asserts that local curvature control at a given time extends to a smaller backward parabolic neighborhood.
It can be thought of as the parabolic analogue of regularity of Einstein metrics (e.g. \cref{cor:curvature-regularity}).

The idea of backward pseudolocality dates back to Perelman's work on Ricci flow with surgery in dimension three \cite[Proposition 6.4]{perelman2003ricci}.
Later in \cite{chen2012space}, Chen--Wang formulated and proved backward pseudolocality for noncollapsed Ricci flows with bounded scalar curvature and $L^{n/2}$-curvature energy, with applications to K\"ahler--Ricci flow on Fano surfaces; their subsequent work \cite{chen2020space} developed two-sided and long-time pseudolocality in the polarized K\"ahler setting.
The result \cite[Theorem 1.5]{bamler2015heat} of Bamler--Zhang removes the curvature energy assumption and gives backward pseudolocality under a local scalar curvature upper bound and an initial entropy bound.
In the more recent work \cite{bamler2020structure}, Bamler proved backward pseudolocality for compact Ricci flows using only local curvature and volume noncollapsing assumptions on the terminal time-slice, without a prescribed spacetime scalar curvature bound.

For our purpose of establishing backward pseudolocality for $G_2$-Laplacian flow with uniformly bounded torsion, the seemingly most promising approach is the one in \cite{bamler2015heat}.
Unfortunately, their construction of auxiliary functions relies on the special structure of the evolution equations of Ricci flow and the existence of a substitute for general Laplacian flow remains unclear.\footnote{We mention that, however, for the special case of hypersymplectic flow (a special four-dimensional reduction of $G_2$-Laplacian flow introduced in \cite{fine2018hypersymplectic}), there is indeed such a counterpart \cite{fine2026convergence}.}
The proof of Bamler in \cite{bamler2020structure} depends on his compactness theorem \cite{bamler2020compactness} and structure theory for limits of Ricci flows \cite{bamler2020structure}.
While the former paper carries over to our case up to minor modifications (which is clear in view of \cite{he2026pseudolocality}), many arguments in the latter paper depend on Perelman's Harnack inequality which is not available for $G_2$-Laplacian flow.
Thus although the validity of backward pseudolocality is believed by experts, a rigorous justification remains technically challenging.

Our proof of backward pseudolocality depends on a new estimate concerning the time-derivative of pointed Nash entropy, as well as certain necessary inputs from \cite{he2026pseudolocality}.
We separate backward pseudolocality from \cite{he2026pseudolocality} because of the usage of LLM in the current paper (see \S\ref{sec:usage-of-LLM} for more information).

\subsection{Structure of the paper}
The notation and convention we adopt are summarized in \S\ref{sec:notation-and-convention}.
We state the main results of this paper in \S\ref{sec:statement-of-main-results} and discuss the main ideas in \S\ref{sec:main-ideas}.
In \cref{sec:estimate-on-Nash-entropy}, we prove a new estimate on $\partial_t\mathcal N_{x,t}(\tau)$ and establish improved curvature estimates for almost static geometric flow (\cref{lem:almost-static}(\ref{lem-it:almost-static-curvature-bound})).
In \cref{sec:proof-of-main-theorem}, we combine the previous estimates to prove our backward pseudolocality theorem.
Several analytic estimates and entropy rigidity results that are established in \cite{he2026pseudolocality} and needed in this paper are collected in Appendix \ref{sec:preliminaries}.
Finally, we discuss the backward pseudolocality of Ricci flow in Appendix \ref{sec:Ricci-flow}.

\subsection{Statement of main results}\label{sec:statement-of-main-results}
The main result of this paper is the following backward pseudolocality theorem.
\begin{Theorem}[Backward pseudolocality]\label{thm:backward-pseudolocality}
	For every $\Lambda,\alpha>0$, there exists $\epsilon(\Lambda,\alpha)>0$ such that the following holds.

	Let $(M^7,\varphi(t),g(t))_{t\in I}$ be a Laplacian flow on a compact manifold, $(x_0,t_0)\in M\times I$ be a point, and $r>0$ be a scale such that $[t_0-r^2,t_0]\subset I$.
	Suppose
	\begin{equation*}
		\sup_{M\times[t_0-r^2,t_0]}|T_{\varphi}|^2\leq\Lambda r^{-2},\qquad
		|B(x_0,t_0,r)|\geq\alpha r^7,\qquad
		|\Rm|\leq(\alpha r)^{-2}\text{ over }B(x_0,t_0,r).
	\end{equation*}
	Then $|\Rm|\leq(\epsilon r)^{-2}$ over $P(x_0,t_0;(1-\alpha)r,-(\epsilon r)^2)$.
\end{Theorem}

Combining \cref{thm:backward-pseudolocality} with the Shi-type estimates, we obtain the following regularity result.
\begin{Corollary}\label{cor:curvature-regularity}
	Under the same assumptions of \cref{thm:backward-pseudolocality}, we have
	\begin{equation*}
		|\nabla^k\Rm|\leq C_k(\Lambda,\alpha)r^{-2-k}\qquad\text{over }B(x_0,t_0,\tfrac{1}{2}r)
	\end{equation*}
	for all $k\geq 0$.
\end{Corollary}

The two-sided pseudolocality theorem follows easily from \cref{thm:backward-pseudolocality} and \cref{thm:refined-pseudolocality}.
Its Ricci flow counterpart is stated, for example, in \cite[Theorem 2.14]{fangli2026strong}.
\begin{Theorem}[Two-sided pseudolocality]\label{thm:two-sided-pseudolocality}
	For every $\Lambda,\alpha>0$, there exists
	$\epsilon=\epsilon(\Lambda,\alpha)$ such that the following holds.

	Let $(M^7,\varphi(t),g(t))_{t\in I}$ be a Laplacian flow on a compact manifold, $(x_0,t_0)\in M\times I$ be a point, and $r>0$ be a scale such that $[t_0-r^2,t_0]\subset I$.
	Suppose
	\begin{equation*}
		\sup_{M\times([t_0-r^2,t_0+r^2]\cap I)}|T_\varphi|^2\leq\Lambda r^{-2},\qquad
		|B(x_0,t_0,r)|\geq\alpha r^7,\qquad
		|\Rm|\leq(\alpha r)^{-2}\text{ over }B(x_0,t_0,r).
	\end{equation*}
	Then $|\Rm|\leq(\epsilon r)^{-2}$ on $P(x_0,t_0;\epsilon r)$.
\end{Theorem}

\subsection{Main ideas}\label{sec:main-ideas}
Let us first restrict ourselves to Ricci flow for motivation and intuition of forward/backward pseudolocality.

Forward pseudolocality is best viewed as a form of rigidity phenomenon concerning entropy functional like Perelman's $\mu$-entropy.
It is also equivalent to the quantitative rigidity of pointed Nash entropy.
See \cite{wang2018local,wang2020local,he2026pseudolocality} for more explanation.

The intuition of backward pseudolocality is that, if the scalar curvature is identically zero, then $\square R=2|\Ric|^2$ implies $\Ric=0$ and the flow is static.
One can think of backward pseudolocality as a quantitative version of the above ``scalar curvature rigidity''.
In differential geometry, a powerful mechanism to handle such quantitative rigidity phenomenon is the compactness argument: one tries to extract a convergent subsequence from an exceptional sequence and apply the rigidity result to the limit.
(This is exactly how Bamler proves backward pseudolocality in \cite{bamler2020structure}.)
Usually, the compactness argument can be replaced by direct quantitative estimates, and this is the origin of the paper.

The key to our direct approach towards backward pseudolocality is a new analytic estimate on $\partial_t\mathcal N_{x,t}(\tau)$.
It allows the smallness of $\mathcal N_{x,t}(\tau)$ to propagate along worldline when combined with the epsilon-regularity result concerning pointed Nash entropy.
Another application of epsilon-regularity then gives the required curvature bound.

\subsection{Notation and convention}\label{sec:notation-and-convention}
\begin{itemize}
	\item (Constants). We allow $C$ to differ from line to line.
	The notation $C(a\mid b)$ denotes a generic function that satisfies $\lim_{a\to 0}C(a\mid b)=0$ for fixed $b$.
	\item (Interval $I$). In this paper, $I$ will always be an interval on which the smooth family of Riemannian metrics is defined.
	\item (Deviation $E$). For such a family $(g(t))_{t\in I}$ on $M^n$, we write
	\begin{equation*}
		E(g):=\partial_tg+2\Ric(g)
	\end{equation*}
	for its deviation from Ricci flow.
	\item (Torsion $T_{\varphi}$). The torsion of a $G_2$-structure $\varphi$ is denoted by $T_{\varphi}$.
	\item (Metric balls). Denote $B(x,t,r):=B_{g(t)}(x,r)$.
	\item (Parabolic neighborhoods). We follow the convention in \cite{bamler2020entropy} of parabolic neighborhoods.
	Namely, for $(x,t)\in M\times I$ and $r,T^+,T^-\geq 0$, we write
	\begin{equation*}
		P(x,t;r,-T^{-},T^{+})
		:=B(x,t,r)\times([t-T^{-},t+T^{+}]\cap I)
	\end{equation*}
	and we may omit $-T^{-}$ or $T^{+}$ if it is zero.
	We will also use the abbreviations
	\begin{align*}
		P(x,t;r)&:=P(x,t;r,-r^2,r^2),  \\
		P^+(x,t;r)&:=P(x,t;r,0,r^2),  \\
		P^-(x,t;r)&:=P(x,t;r,-r^2,0).
	\end{align*}
	\item (Heat operator $\square$). Denote $\square:=\partial_t-\Delta$.
	\item (Heat kernel $K$). Let $M^n$ be a compact $n$-manifold and $(g(t))_{t\in I}$ be a smooth family of Riemannian metrics.
	For any $x,y\in M$ and $s,t\in I$ with $s<t$, we denote by $K(x,t;y,s)$ the heat kernel of $\square$.
	\item (Conjugate heat kernel measure $\nu$). We denote by $\nu_{x,t;s}$ the conjugate heat kernel measure based at $(x,t)$, i.e.,
	\begin{equation*}
		d\nu_{x,t;s}:=K(x,t;\cdot,s)\,dV_{g(s)},\qquad
		\nu_{x,t;t}:=\delta_x.
	\end{equation*}
	\item (The operator $P_{s,t}$). Given $\phi\in C^\infty(M)$, we define
	\begin{equation*}
		P_{s,t}\phi(x):=\int_M  K(x,t;y,s)\phi(y)  \,dV_s(y).
	\end{equation*}
	\item (Ricci curvature integral $\mathcal{I}$). This is defined in \cref{eq:integral-of-Ricci-curvature}.
\end{itemize}
We also recall some basic properties of $P_{s,t}$.
When applied to $\phi\in C^\infty(M)$, we have
\begin{equation}\label{eq:basic-properties-of-heat-propagator}
	P_{s,t}=P_{a,t}\circ P_{s,a},\qquad
	\partial_tP_{s,t}=\Delta_tP_{s,t},\qquad
	\partial_sP_{s,t}=-P_{s,t}\Delta_s
\end{equation}
for all $s<a<t$.
A direct consequence of the last identity is Duhamel's formula
\begin{equation}\label{eq:Duhamel-formula}
	u_t-P_{s,t}u_s
	=\int_s^t  P_{a,t}\square_au_a  \,da,\qquad
	\forall\,u\in C^\infty(M\times I).
\end{equation}

\subsection{Acknowledgments}
Both author thank Professor Bing Wang and Professor Yu Li for helpful discussions on Ricci flow.
The second author is also grateful to Professor Bing Wang for his guidance and constant support.

\subsection{Usage of LLM}\label{sec:usage-of-LLM}
During the preparation of this paper, we used GPT-6 Astra to assist with literature searches, exploration of proof strategies, checking calculations, and refinement of the exposition.
The author takes full responsibility for the mathematical arguments, references, and final text.

\section{An estimate on pointed Nash entropy}\label{sec:estimate-on-Nash-entropy}
In this section, we establish a new estimate on pointed Nash entropy that is key to our proof of backward pseudolocality.
Strictly speaking, it is a combination of the following results.
\begin{itemize}
	\item \cref{prop:derivative-of-Nash-entropy-in-time-variable}: bounding $\partial_t\mathcal N_{x,t}(\tau)$ by the integral of Ricci curvature $\mathcal I$ (defined in \cref{eq:integral-of-Ricci-curvature}).
	\item \cref{lem:small-Ricci-curvature-integral}: smallness of $\mathcal I$ for almost static flow in the presence of suitable cut-off.
	\item \cref{lem:almost-static}(\ref{lem-it:existence-of-cut-off}): existence of cut-off for sufficiently regular almost static flow.
\end{itemize}

Let us start from the commutator of Laplacian and $P_{s,t}$.
It appears naturally in the quantity $\partial_t\mathcal N_{x,t}(\tau)$ and is of independent interest.
Define
\begin{equation}\label{eq:integral-of-Ricci-curvature}
	\mathcal I(x,t;\tau)
	:=\int_{t-\tau}^t  P_{a,t}|\Ric|^2(x)  \,da
	=\int_{t-\tau}^t\int_M  |\Ric|^2  \,d\nu_{x,t;a}\,da.
\end{equation}

\begin{Proposition}\label{prop:commutator}
	Let $(M^n,g(a))_{a\in I}$ be a smooth family of Riemannian metrics on a compact manifold and $s<t$ be such that $[s,t]\subset I$.
	Denote $\tau:=t-s$.
	Suppose $|E(g)|\leq\Lambda$ and $R(g)\geq-\Lambda$ on $M\times[s,t]$ for some $\Lambda>0$.
	Then for any $\phi\in C^\infty(M)$, we have
	\begin{equation*}
		|\Delta_tP_{s,t}\phi-P_{s,t}\Delta_s\phi|
		\leq\left(  \sqrt{2}e^{\Lambda\tau/2}\mathcal I(\cdot,t;\tau)^{1/2}  +\frac{C(\Lambda\tau\mid n)}{\sqrt{\tau}}  \right)
		\left(  P_{s,t}|\nabla\phi|_s^2  \right)^{1/2}.
	\end{equation*}
\end{Proposition}
\begin{proof}
	Fix $x\in M$.
	Set $u_a:=P_{s,a}\phi$ and define $f$ by $K(x,t;y,a)=(4\pi(t-a))^{-n/2}e^{-f(y,a)}$.
	Applying \cref{eq:Duhamel-formula} to $\Delta u$ yields
	\begin{equation}\label{eq:Duhamel-formula-applied-to-Laplacian-u}
		\Delta_tP_{s,t}\phi-P_{s,t}\Delta_s\phi
		=\int_s^t  P_{a,t}\square_a(\Delta_au_a)  \,da.
	\end{equation}
	Since $\square u=0$, we have
	\begin{equation*}
		\square(\Delta u)
		=2\langle  \Ric,\nabla^2u  \rangle
		-\langle  E,\nabla^2u  \rangle
		+\langle  \delta E+\tfrac{1}{2}\nabla\Tr E,\nabla u  \rangle.
	\end{equation*}
	Substituting this into \cref{eq:Duhamel-formula-applied-to-Laplacian-u} and integrating by parts, we obtain
	\begin{multline}\label{eq:heat-Laplacian-commutator}
		\Delta_tP_{s,t}\phi-P_{s,t}\Delta_s\phi
		=2\int_s^t  P_{a,t}\langle  \Ric_a,\nabla^2u_a  \rangle_a  \,da  \\
		-\int_s^t  P_{a,t}\Big(  \left(E_a-\tfrac{1}{2}(\Tr E)_ag_a\right)(\nabla f_a,\nabla u_a)+\tfrac{1}{2}(\Tr E)_a\Delta_au_a  \Big)\,da.
	\end{multline}
	Denote $F(a):=P_{a,t}|\nabla u_a|^2(x)$.
	By \cref{eq:basic-properties-of-heat-propagator},
	\begin{align}
		F'(a)=P_{a,t}\square_a|\nabla u_a|^2(x)
		&=P_{a,t}\left(-2|\nabla^2u_a|^2-E(\nabla u_a,\nabla u_a)\right)(x)  \notag  \\
		&\leq-2P_{a,t}|\nabla^2u_a|^2(x)+\Lambda F(a)  \label{ineq:Bochner-formula}.
	\end{align}
	In particular, $F'\leq\Lambda F$ and therefore
	\begin{equation}\label{ineq:estimate-on-F}
		F(a)\leq e^{\Lambda(a-s)}F(s)\leq e^{\Lambda\tau}P_{s,t}|\nabla\phi|_s^2(x)
	\end{equation}
	for all $a\in[s,t]$.
	Moreover, integrating \cref{ineq:Bochner-formula} and using \cref{ineq:estimate-on-F}, we obtain
	\begin{equation}\label{ineq:estimate-on-hessian}
		2\int_s^t  P_{a,t}|\nabla^2u_a|^2(x)  \,da
		\leq e^{\Lambda\tau}P_{s,t}|\nabla\phi|_s^2(x).
	\end{equation}
	The assertion then follows by applying Cauchy--Schwarz inequality to \cref{eq:heat-Laplacian-commutator} and using the bounds \cref{ineq:estimate-on-F}, \cref{ineq:estimate-on-hessian}, \cref{ineq:energy-estimate}.
\end{proof}

The following proposition controls the time-derivative of pointed Nash entropy by the Ricci curvature integral $\mathcal I$.
\begin{Proposition}\label{prop:derivative-of-Nash-entropy-in-time-variable}
	In the setting of \cref{prop:commutator}, we have
	\begin{equation}\label{eq:derivative-of-Nash-entropy-in-time-variable}
		\partial_t\mathcal N_{x,t}(\tau)
		\leq\frac{\sqrt n+C(\Lambda\tau\mid n)}{\sqrt\tau}\mathcal I(x,t;\tau)^{1/2}
		+\frac{C(\Lambda\tau\mid n)}{\tau}
	\end{equation}
	for all $x\in M$.
\end{Proposition}
\begin{proof}
	Define $f_{x,t;s}(y)$ by $K(x,t;y,s)=(4\pi\tau)^{-n/2}e^{-f_{x,t;s}(y)}$.
	Recall the pointed Nash entropy is defined by $\mathcal N_{x,t}(\tau)=(P_{s,t}f_{x,t;s})(x)-\frac{n}{2}$.
	Denote $\dot{f}_{x,t;s}:=\frac{d}{dh}|_{h=0}f_{x,t+h;s+h}$.
	Then \cref{eq:basic-properties-of-heat-propagator} gives
	\begin{align*}
		\partial_t\mathcal N_{x,t}(\tau)
		&=\frac{d}{dh}\Big|_{h=0}(P_{s+h,t+h}f_{x,t+h;s+h})(x)  \\
		&=(-P_{s,t}\Delta_sf_{x,t;s}+\Delta_tP_{s,t}f_{x,t;s}+P_{s,t}\dot{f}_{x,t;s})(x).
	\end{align*}
	Differentiating
	\begin{equation*}
		(4\pi\tau)^{-n/2}\int_M  e^{-f_{x,t;s}}  \,dV_{s}=1,
	\end{equation*}
	we get $(P_{s,t}\dot{f}_{x,t;s})(x)=(-P_{s,t}(R-\frac{1}{2}\Tr E)_s)(x)$.
	The assertion then follows from \cref{prop:commutator}, \cref{ineq:energy-estimate} and the assumed bound on $E$ and $R$.
\end{proof}

\begin{Remark}
	Clearly, if we assume two-sided bound on scalar curvature (i.e. $|R(g)|\leq\Lambda$), then the bound in \cref{eq:derivative-of-Nash-entropy-in-time-variable} improves to a bound on $|\partial_t\mathcal N_{x,t}(\tau)|$.
	But \cref{eq:derivative-of-Nash-entropy-in-time-variable} is already enough to propagate an entropy lower bound backward (see the proof of \cref{lem:extension}).
\end{Remark}

Along a Ricci flow $(M^n,g(t))_{t\in I}$, we have $\square R=2|\Ric|^2$ and therefore \cref{eq:Duhamel-formula} gives
\begin{equation}\label{eq:Ricci-curvature-integral-in-the-case-of-Ricci-flow}
	2\int_s^t  P_{a,t}|\Ric|^2(x)  \,da
	=R(x,t)-P_{s,t}R_s(x).
\end{equation}
So smallness of the scalar curvature implies smallness of the integral of Ricci curvature $\mathcal I$.
The following lemma generalizes this under suitable assumptions.

\begin{Lemma}\label{lem:small-Ricci-curvature-integral}
	For any $n\in\N$ and $Y,L,\lambda,\kappa\in[0,\infty)$, there exists a constant $C(\lambda,\kappa\mid n,Y,L)<\infty$ such that the following holds.

	Let $(M^n,g(t))_{t\in I}$ be a smooth family of Riemannian metrics on a compact manifold, $(x_0,t_0)\in M\times I$ be a point, and $r>0$ be a scale such that $[t_0-r^2,t_0]\subset I$.
	Suppose $\mathcal N_{x_0,t_0}(r^2)\geq-Y$ and
	\begin{equation*}
		|E(g)|\leq\lambda r^{-2},\qquad
		R(g)\geq-\lambda r^{-2}
	\end{equation*}
	over $M\times[t_0-r^2,t_0]$.
	Assume $\chi:M\times[t_0-r^2,t_0]\to[0,1]$ is a smooth function such that
	\begin{equation*}
		\chi(x_0,t_0)=0,\qquad
		r^2|\partial_t\chi|+r|\nabla\chi|+r^2|\nabla^2\chi|\leq L,\qquad
		|\Ric|\leq\kappa r^{-2}\text{ over }\{\chi<1\}.
	\end{equation*}
	Then
	\begin{equation*}
		\int_{t_0-r^2}^{t_0}\int_M  |\Ric|^2  \,d\nu_{x_0,t_0;t}\,dt
		\leq C(\lambda,\kappa\mid n,Y,L)r^{-2}.
	\end{equation*}
\end{Lemma}
\begin{proof}
	By time-shifting and parabolic rescaling, we may assume without loss of generality that $t_0=0$ and $r=1$.
	Denote $d\nu_t:=d\nu_{x_0,t_0;t}$ and let $f$ be defined by $d\nu=(4\pi\tau)^{-n/2}e^{-f}dV$.
	Set
	\begin{equation*}
		F(t):=\int_M  \chi(x,t)R(x,t)  \,d\nu_t(x).
	\end{equation*}
	Using $\square R=2|\Ric|^2-\langle\Ric,E\rangle+\delta^2E-\Delta\Tr E$ and integrating by parts give
	\begin{equation}\label{ineq:derivative-of-local-scalar-curvature-integral}
		\begin{split}
			F'(t)=&\int_M  \big(  \chi\square R  +(\square\chi)R  -2\langle\nabla\chi,\nabla R\rangle  \big)  \,d\nu_t  \\
			=&\;2\int_M\chi|\Ric|^2\,d\nu_t
			-\int_M\chi\langle E,\Ric\rangle\,d\nu_t
			+\int_M\chi\langle E,\Phi\rangle\,d\nu_t
			+\Psi_2(t)+\Psi_1(t),
		\end{split}
	\end{equation}
	where ($K$ denotes the conjugate heat kernel based at $(x_0,t_0)$)
	\begin{align*}
		\Phi&:=K^{-1}\big(  \nabla^2K-(\Delta K)g  \big),  \\
		\Psi_2(t)&:=\int_M\big(  \langle E,\nabla^2\chi\rangle-(\Tr E)\Delta\chi+R(\partial_t\chi+\Delta\chi)  \big)\,d\nu_t,  \\
		\Psi_1(t)&:=-\int_M\big(  2E(\nabla\chi,\nabla f)+2(R-\Tr E)\langle\nabla\chi,\nabla f\rangle  \big)\,d\nu_t.
	\end{align*}
	Our assumptions imply
	\begin{equation}\label{ineq:direct-estimates}
		\int_M  \chi\langle E,\Ric\rangle  \,d\nu_t
		\leq\tfrac{1}{4}\lambda^2+\int_M  \chi|\Ric|^2  \,d\nu_t,
	\end{equation}
	\begin{equation*}
		|\Psi_2|\leq C(n)L(\lambda+\kappa),\qquad
		|\Psi_1(t)|\leq\frac{C(\lambda,\kappa\mid n,L)}{\sqrt{-t}}.
	\end{equation*}
	Here we use \cref{ineq:energy-estimate} in the last inequality.
	Note also $F(0)=0$, $F(-1)\geq-\lambda$.
	Integrating \cref{ineq:derivative-of-local-scalar-curvature-integral} and substituting the above estimates yield
	\begin{equation*}
		\int_{-1}^{0}\int_M  \chi|\Ric|^2  \,d\nu_t\,dt
		\leq C(\lambda,\kappa\mid n,L)+\lambda\int_{-1}^{0}\int_M\chi|\Phi|\,d\nu_t\,dt.
	\end{equation*}

	We estimate the last integral involving $\Phi$ as follows.
	A direct computation gives $\Phi=-\nabla^2f+df\otimes df+(\Delta f-|\nabla f|^2)g$.
	Moreover, an integration by parts gives
	\begin{equation*}
		\int_M  \chi\left(|\nabla f|^2-\Delta f\right)  \,d\nu_t
		=\int_M  \chi\Delta K  \,dV_t
		=\int_M  \Delta\chi  \,d\nu_t.
	\end{equation*}
	It follows that
	\begin{equation}\label{eq:intermediate-estimate-on-Phi}
		\int_M\chi|\Phi|\,d\nu_t
		\leq C(n)\int_M  \chi\big(  |\nabla^2f|+|\nabla f|^2  \big)  \,d\nu_t
		\leq C(n)\int_M  \chi|\nabla^2f|  \,d\nu_t  +C(n)L.
	\end{equation}
	Denote $\Theta:=\Ric+\nabla^2f-\frac{g}{2\tau}-\frac{E}{4}$.
	Then \cite[(3.7)]{chen2018shi} gives
	\begin{equation*}
		\int_{-1}^{0}\int_M  \left(  2\tau|\Theta|^2-\tfrac{\tau}{8}|E|^2  \right)  \,d\nu_t\,dt
		=\mathcal W_{x_0,t_0}(\tau)|_{t=-1}^0.
	\end{equation*}
	One can show $\lim_{\tau\to 0}\mathcal W_{x_0,t_0}(\tau)=0$ as in \cite{wang2018local}.
	By the definition of $\mathcal W$ and $\mathcal N$, we have
	\begin{equation*}
		\mathcal W_{x_0,t_0}(1)
		=\mathcal N_{x_0,t_0}(1)
		+\int_M\big(|\nabla f|^2+R\big)  \,d\nu_{-1}-\frac{n}{2}
		\geq -Y-\lambda-\frac{n}{2}.
	\end{equation*}
	Therefore
	\begin{equation*}
		\int_{-1}^{0}\int_M  \tau|\Theta|^2  \,d\nu_t\,dt\leq C(n,Y)+C(\lambda\mid n,Y).
	\end{equation*}
	In order to bound the integral of \cref{eq:intermediate-estimate-on-Phi} over $[-1,0]$, we use $|\nabla^2f|\leq C(n)\left(  |\Theta|+|\Ric|+\tfrac{1}{\tau}+\lambda  \right)$.
	The term involving $\Ric$ can be treated (absorbed) as in \cref{ineq:direct-estimates}.
	Consider $m(t):=\int_{M}\chi(x,t)\,d\nu_t(x)$.
	Then it is easy to see
	\begin{equation*}
		m(0)=0,\qquad
		|m'(t)|=\bigg|\int_{M}\square\chi\,d\nu_t\bigg|\leq C(n)L,\qquad
		0\leq m(t)\leq C(n)L\tau.
	\end{equation*}
	This handles the term $\frac{1}{\tau}$.
	Moreover, it also gives
	\begin{equation*}
		\int_{-1}^{0}\int_M  \chi|\Theta|  \,d\nu_t\,dt
		\leq\int_{-1}^{0}\sqrt{m(t)}\left(\int_M\chi|\Theta|^2\,d\nu_t\right)^{1/2}\,dt
		\leq C(n,L)\left(  \int_{-1}^{0}\int_M  \tau|\Theta|^2  \,d\nu_t\,dt  \right)^{1/2}.
	\end{equation*}
	This completes the proof of
	\begin{equation*}
		\int_{-1}^{0}\int_M  \chi|\Ric|^2  \,d\nu_t\,dt
		\leq C(\lambda,\kappa\mid n,Y,L).
	\end{equation*}
	The assertion then follows by writing $1=\chi+(1-\chi)$ and using $|\Ric|\leq\kappa$ over the set $\{\chi<1\}$.
\end{proof}

Later we shall need the following interpolation inequality.
Its Euclidean version is given, for example, in \cite{nirenberg1966extended}.
In our application, we actually have injectivity radius lower bound and bounds on higher order derivatives of curvature, so the interpolation inequality can be obtained by applying the Euclidean one to normal coordinates.
Below we also give a simple direct proof.
\begin{Lemma}\label{lem:interpolation}
	Let $(M^n,g)$ be a complete Riemannian manifold, $E$ be a tensor bundle on $M$, $x_0\in M$ and $r>0$.
	We write $B_\rho:=B(x_0,\rho)$ for simplicity.
	Suppose
	\begin{equation*}
		|\Rm|\leq K\qquad\text{on }B_{2r}
	\end{equation*}
	for some $K<\infty$.
	Then for any $T\in C^4(B_{2r},E)$, we have
	\begin{equation}\label{ineq:interpolation}
		\|\nabla^2T\|_{L^\infty(B_r)}
		\leq C(n,\mathrm{rank}E)\left(  \|T\|_{L^\infty(B_{2r})}^{1/2}  \|\nabla^4T\|_{L^\infty(B_{2r})}^{1/2}  +(r^{-2}+K)\|T\|_{L^\infty(B_{2r})}  \right).
	\end{equation}
\end{Lemma}

\begin{proof}
	To simplify notation, we write $A_0:=\|T\|_{L^\infty(B_{2r})}$ and $A_4:=\|\nabla^4T\|_{L^\infty(B_{2r})}$.
	Let $x\in B_r$ be a point, $v\in T_xM$ be a unit vector, and $\ell\in(0,r/2]$.
	Consider $\gamma(s):=\exp_x(sv)$ for $s\in[-\ell,\ell]$ and define $F(s):=P_{\gamma}(T(\gamma(s)))$, where $P_{\gamma}$ denotes parallel transport along
	$\gamma$ to $x$.
	Then $F$ is of the class $C^4$ and $F''(0)=(\nabla^2_{v,v}T)_x$, $|F(s)|\leq A_0$, $|F^{(4)}(s)|\leq A_4$.
	Taylor's formula at $s=0$, applied at $s=\pm\ell$, gives
	\begin{equation*}
		F(\ell)+F(-\ell)-2F(0)=\ell^2F''(0)+\mathcal E,\qquad
		|\mathcal E|\leq\frac{\ell^4}{12}A_4.
	\end{equation*}
	It follows that $|\nabla^2_{v,v}T|\leq 4\ell^{-2}A_0+\frac{\ell^2}{12}A_4$ and polarization therefore yields
	\begin{equation*}
		|(\nabla^2_{X,Y}+\nabla^2_{Y,X})T|
		\leq C(n)(\ell^{-2}A_0+\ell^2A_4)
	\end{equation*}
	for unit vectors $X,Y\in T_xM$.
	Combining this with the Ricci identity $(\nabla^2_{X,Y}-\nabla^2_{Y,X})T=\Rm_{X,Y}T$ and using the assumed curvature bound, we obtain
	\begin{equation*}
		\|\nabla^2T\|_{L^\infty(B_r)}
		\leq C(n,\mathrm{rank}E)\big((\ell^{-2}+K)A_0+\ell^2A_4\big).
	\end{equation*}
	Then \cref{ineq:interpolation} follows by taking
	\begin{equation*}
		\ell=\min\left\{\frac r2,\left(\frac{A_0}{A_4}\right)^{1/4}\right\},
	\end{equation*}
	where the right-hand-side is understood to be $r/2$ when $A_4=0$.
\end{proof}

The next lemma provides the cut-off function needed in \cref{lem:small-Ricci-curvature-integral}.
It also supplies improved higher order curvature estimates for future usage.
\begin{Lemma}\label{lem:almost-static}
	For any $n\in\N$, $v,\lambda>0$ and $A<\infty$, there are $\alpha(n,A)>0$, $C(\lambda\mid n,A)<\infty$ and $L(n,A,v)<\infty$ such that the following holds.

	Let $(M^n,g(t))_{t\in I}$ be a smooth family of Riemannian metrics on a compact manifold, $(x_0,t_0)\in M\times I$ be a point, and $r>0$ be a scale such that $[t_0-r^2,t_0]\subset I$.
	Suppose
	\begin{equation*}
		|E(g)|+|R(g)|\leq\lambda r^{-2},\qquad
		\sum_{j=0}^4r^{2+j}\big(|\nabla^j\Rm|+|\nabla^jE|\big)\leq A
	\end{equation*}
	over $P^{-}(x_0,t_0;r)$.
	\begin{enumerate}
		\item \label{lem-it:almost-static-curvature-bound}
		Over $P^{-}(x_0,t_0;\alpha r)$, we have $|\Ric|\leq C(\lambda\mid n,A)r^{-2}$ and
		\begin{equation}\label{ineq:curvature-Lipschitz-in-time-variable}
			\big||\Rm|(y,t)-|\Rm|(y,s)\big|
			\leq C(\lambda\mid n,A)r^{-4}|t-s|.
		\end{equation}

		\item \label{lem-it:existence-of-cut-off}
		Assume in addition that $|B(x_0,t_0,r)|\geq vr^n$.
		Then there exists a smooth function $\chi:M\times[t_0-r^2,t_0]\to[0,1]$ such that
		\begin{equation*}
			\chi(x_0,t_0)=0,\qquad
			\{\chi<1\}\subset P^{-}(x_0,t_0;\alpha r),\qquad
			r^2|\partial_t\chi|+r|\nabla\chi|+r^2|\nabla^2\chi|\leq L.
		\end{equation*}
	\end{enumerate}
\end{Lemma}

\begin{proof}
	By time-shifting and parabolic rescaling, we may assume $t_0=0$ and $r=1$.
	Set $h:=\partial_tg$.
	Clearly, $\sum_{j=0}^4|\nabla^jh|\leq C(n,A)$ over $P^{-}(x_0,0;1)$.
	Then the interpolation inequality \cref{ineq:interpolation} gives
	\begin{equation*}
		|\nabla^2R|+|\nabla^2E|\leq C\lambda^{1/2}\qquad\text{over }P^{-}(x_0,0;4\alpha)
	\end{equation*}
	for some $\alpha>0$.
	Integrating $\square R=2|\Ric|^2-\langle\Ric,E\rangle+\delta^2E-\Delta\Tr E$, we have
	\begin{align}
		2\int_{t-\sigma}^{t}  |\Ric|^2(y,s)  \,ds
		&=R(y,t)-R(y,t-\sigma)
		-\int_{t-\sigma}^t  \big(\Delta R  -\langle \Ric,E  \rangle  +\delta^2E  -\Delta\Tr E  \big)(y,s)  \,ds  \notag  \\
		&\leq 2\lambda+C\lambda^{1/2}\sigma  \label{ineq:integral-of-Ricci-curvature}
	\end{align}
	for $(y,t)\in P^{-}(x_0,0;2\alpha)$ and $\sigma\in[0,\alpha^2]$.
	The evolution equation of curvature \cite[\S 3.3]{lotay2017laplacian} gives
	\begin{equation}\label{ineq:derivative-of-Ricci-curvature}
		|\partial_t\Ric|\leq C(n)\big(|\nabla^2h|+|\Rm|\,|h|\big)\leq C(n,A),\qquad
		\left|\frac{d}{dt}|\Ric|^2(y,t)\right|\leq C(n,A)
	\end{equation}
	over $P^{-}(x_0,0;1)$.
	If $\lambda>\alpha^4$, then $|\Ric|\leq C\lambda^{1/4}$ on $P^{-}(x_0,0;2\alpha)$ automatically by the uniform curvature bound.
	So we may assume $\lambda\in(0,\alpha^4]$.
	(In the case $\lambda=0$, $g(t)$ is clearly Ricci flat and independent of $t$ over $P^{-}(x_0,0;1)$.)
	For $\sigma:=\lambda^{1/2}$, \cref{ineq:integral-of-Ricci-curvature} shows
	\begin{equation*}
		\frac{1}{\sigma}\int_{t-\sigma}^t|\Ric|^2(y,s)\,ds\leq C\lambda^{1/2}.
	\end{equation*}
	So there is $s_*\in[t-\sigma,t]$ with
	$|\Ric|^2(y,s_*)\leq C\lambda^{1/2}$ and \cref{ineq:derivative-of-Ricci-curvature} implies
	\begin{equation*}
		|\Ric|^2(y,t)\leq|\Ric|^2(y,s_*)+C\sigma\leq C\lambda^{1/2}.
	\end{equation*}
	This shows $|\Ric|\leq C\lambda^{1/4}$ over $P^{-}(x_0,0;2\alpha)$.
	Combining it with \cref{ineq:interpolation} gives $|\nabla^2\Ric|\leq C\lambda^{1/8}$ over $P^{-}(x_0,0,\alpha)$.
	So far we have shown
	\begin{equation*}
		|h|\leq C\lambda^{1/4},\qquad
		|\nabla^2h|\leq C\lambda^{1/8}\quad\text{over }P^{-}(x_0,0,\alpha).
	\end{equation*}
	By the curvature evolution equation \cite[\S 3.3]{lotay2017laplacian}, for each fixed $y\in B(x_0,0,\alpha)$, we have
	\begin{equation*}
		\left|\frac{d}{dt}|\Rm|_{g(t)}\right|
		\leq C\big(|\nabla^2h|+|\Rm|\,|h|\big)
		\leq C(\lambda^{1/8}+\lambda^{1/4}),\qquad\text{a.e. }t\in[-\alpha^2,0].
	\end{equation*}
	Integrating this gives \cref{ineq:curvature-Lipschitz-in-time-variable}.

	Next we construct the cut-off function.
	Choose smooth functions $\zeta:[0,\infty)\to[0,1]$ and $\eta:[-1,0]\to[0,1]$ satisfying
	\begin{equation*}
		\zeta=1\text{ on }[0,1],\qquad
		\zeta=0\text{ on }[4,\infty),\qquad
		|\zeta'|+|\zeta''|\leq C;
	\end{equation*}
	\begin{equation*}
		\eta=0\text{ on }[-1,-\tfrac{1}{2}\alpha^2],\qquad
		\eta=1\text{ on }[-\tfrac{1}{4}\alpha^2,0],\qquad
		|\eta'|\leq C\alpha^{-2}.
	\end{equation*}
	The injectivity radius estimate of Cheeger--Gromov--Taylor implies $\mathrm{inj}_{g(0)}(x_0)>c(n,A,v)>0$.
	Define $\rho:=\frac{1}{2}\min\{c(n,A,v),\alpha\}$ and
	\begin{equation*}
		\chi_0(y):=\zeta\bigg(\frac{d_{g(0)}(x_0,y)^2}{\rho^2}\bigg),\qquad
		\chi(y,t):=1-\eta(t)\chi_0(y).
	\end{equation*}
	Then $|d\chi_0|_{g(0)}\leq C\rho^{-1}$ and $|\nabla^{2,g(0)}\chi_0|_{g(0)}\leq C\rho^{-2}$.
	Moreover, the assumed bounds on $\partial_tg$ and its derivative give
	\begin{equation*}
		|\nabla\chi_0|\leq C\rho^{-1},\qquad
		|\nabla^2\chi_0|\leq C(\rho^{-2}+\alpha^2\rho^{-1}).
	\end{equation*}
	for $t\in[-\alpha^2,0]$.
	It follows that
	\begin{equation*}
		|\partial_t\chi|\leq C\alpha^{-2},\qquad
		|\nabla\chi|\leq C\rho^{-1},\qquad
		|\nabla^2\chi|\leq C(\rho^{-2}+\alpha^2\rho^{-1}).
	\end{equation*}
	Clearly, we also have $\chi(x_0,0)=0$ and $\{\chi<1\}\subset P^{-}(x_0,0;\alpha)$.
\end{proof}

\section{Proof of \texorpdfstring{\cref{thm:backward-pseudolocality}}{backward pseudolocality}}\label{sec:proof-of-main-theorem}
This section combines our preparations in \S\ref{sec:estimate-on-Nash-entropy} to give the final proof of \cref{thm:backward-pseudolocality}.
Here we restrict ourselves to $G_2$-Laplacian flow.
We shall only use the following properties of this flow.
\begin{itemize}
	\item Shi-type estimates \cite[Theorem 4.4]{lotay2017laplacian}.
	\item Compactness theorem \cite[Theorem 7.2]{lotay2017laplacian}.
	\item $|E(g)|\leq C|T_{\varphi}|^2$ for some universal constant $C$ and $R_g=-|T_{\varphi}|^2$.
	So a bound on torsion $T_{\varphi}$ controls both the deviation $E$ and the scalar curvature $R$.
\end{itemize}

\begin{Lemma}\label{lem:extension}
	For every $\alpha>0$, there exists a constant $\lambda(\alpha)>0$ such that the following holds.
	
	Let $(M^7,\varphi(t),g(t))_{t\in I}$ be a Laplacian flow on a compact manifold, $(x_0,t_0)\in M\times I$ be a point, and $r>0$ be a scale such that $[t_0-r^2,t_0]\subset I$.
	Suppose
	\begin{equation*}
		\sup_{M\times[t_0-r^2,t_0]}|T_{\varphi}|^2\leq\lambda r^{-2},\qquad
		|B(x_0,t_0,r)|\geq\alpha r^7,\qquad
		|\Rm|\leq\alpha^{-1}r^{-2}\text{ over }P(x_0,t_0;r,-(\alpha r)^2).
	\end{equation*}
	Then over $P(x_0,t_0;(1-\alpha)r,-(1-\alpha)r^2)$, we have
	\begin{equation*}
		|\Ric|\leq\alpha r^{-2},\qquad
		|\Rm|\leq\alpha r^{-2}+\sup_{B(x_0,t_0,r)}|\Rm|.
	\end{equation*}
\end{Lemma}
\begin{proof}
	By time-shifting and parabolic rescaling, we may assume $t_0=0$ and $r=1$.
	Set
	\begin{equation*}
		A:=\overline{B(x_0,0,1-\tfrac{\alpha}{2})}.
	\end{equation*}
	Let $\epsilon\in(0,\frac{1}{4})$ and $\eta(\epsilon)>0$ be the constant determined in \cref{thm:refined-epsilon-regularity}.
	The assumed volume ratio lower bound and \cref{thm:no-local-inflating} give Nash entropy lower bound.
	Combining it with the curvature bound allows us to apply \cref{thm:entropy-continuity} to choose $\rho>0$ such that
	\begin{equation*}
		\mathcal N_{y,0}(\rho^2)\geq-\eta/4
	\end{equation*}
	for all $y\in A$.
	We claim $\mathcal N_{y,s}(\rho^2)\geq-\eta$ for all $(y,s)\in A\times[-(1-\alpha),0]$.
	To see this, we define
	\begin{equation*}
		t_*:=\inf\left\{t\in[-(1-\alpha),0]:\mathcal N_{y,s}(\rho^2)\geq-\eta\text{ for all }(y,s)\in A\times[t,0]\right\}.
	\end{equation*}
	Then it suffices to show $t_*=-(1-\alpha)$.
	Let $(y,s)\in A\times[t_*,0]$.
	The epsilon-regularity \cref{thm:refined-epsilon-regularity} together with Shi-type estimates supplies the curvature derivative bounds required in \cref{lem:almost-static}, which yields the cut-off function needed in \cref{lem:small-Ricci-curvature-integral}.
	The latter lemma gives smallness of $\mathcal I(y,s;\rho^2)$ and \cref{prop:derivative-of-Nash-entropy-in-time-variable} then yields $\partial_s\mathcal N_{y,s}(\rho^2)\leq C(\lambda\mid\alpha)$.
	By taking $\lambda$ sufficiently small, we have
	\begin{equation*}
		\mathcal N_{y,t_*}(\rho^2)\geq\mathcal N_{y,0}(\rho^2)-\eta/4\geq-\eta/2.
	\end{equation*}
	Were $t_*>-(1-\alpha)$, then the compactness of $A$ and continuity of pointed Nash entropy would make $\mathcal N_{y,s}\geq-\eta$ on a larger interval, contradicting the choice of $t_*$.
	This proves the claim.

	The above claim and \cref{thm:refined-epsilon-regularity} together give curvature bound on a parabolic neighborhood for each point of $P(x_0,0;1-\alpha,-(1-\alpha))$.
	Using the Shi-type estimates and \cref{thm:no-local-collapsing}, we are allowed to apply \cref{lem:almost-static}(\ref{lem-it:almost-static-curvature-bound}) to finish the proof.
\end{proof}

With \cref{lem:extension} in hand, \cref{thm:backward-pseudolocality} follows from an induction argument over scales as in \cite[Lemma 26.4]{bamler2020structure}.
We copy the proof here for completeness.
\begin{Lemma}\label{lem:induction-over-scales}
	For every $\alpha>0$, there exists a constant $\lambda(\alpha)>0$ such that the following holds.

	Let $(M^7,\varphi(t),g(t))_{t\in I}$ be a Laplacian flow on a compact manifold, $(x_0,t_0)\in M\times I$ be a point, and $r>0$ be a scale such that $[t_0-r^2,t_0]\subset I$.
	Suppose
	\begin{enumerate}
		\item $|B(x_0',t_0,r')|\geq\alpha{r'}^7$ for any ball $B(x_0',t_0,r')\subset B(x_0,t_0,r)$;
		\item $|\Rm|\leq\alpha^{-1}r^{-2}$ on $B(x_0,t_0,r)$;
		\item $|T_{\varphi}|^2\leq\lambda r^{-2}$ over $M\times[t_0-r^2,t_0]$.
	\end{enumerate}
	Then $|\Rm|\leq 2\alpha^{-1}r^{-2}$ over $P(x_0,t_0;\frac{1}{2}r,-\frac{1}{2}r^2)$.
\end{Lemma}
\begin{proof}
	Fix $\alpha\in(0,\frac{1}{2})$ and choose $\lambda(\frac{1}{8}(\frac{4}{3})^2\alpha)$ as in \cref{lem:extension}.
	Assume that the lemma was false.
	After possibly replacing $x_0,r$ with some $x_0',r'$ such that $B(x_0',t_0,r')\subset B(x_0,t_0,r)$, we may assume that the lemma holds for any such point and scale if $r'\leq\frac{1}{2}r$.
	Since the assumptions of the lemma continue to hold after such a replacement, we obtain that for any $B(x_0',t_0,r')\subset B(x_0,t_0,r)$ with $r'\leq\frac{1}{2}r$ we have $|\Rm|\leq 2\alpha^{-1}{r'}^{-2}$ on $P(x_0',t_0;\frac{1}{2}r',-\frac{1}{2}{r'}^2)$.
	This implies that
	\begin{equation*}
		|\Rm|\leq 8\alpha^{-1}r^{-2}=(\tfrac{1}{8}(\tfrac{4}{3})^2\alpha)^{-1}(\tfrac{3}{4}r)^{-2}
		\qquad\text{on }P(x_0,t_0;\tfrac{3}{4}r;-\tfrac{1}{8}r^2).
	\end{equation*}
	Applying \cref{lem:extension} to the ball $B(x_0,t_0,\frac{3}{4}r)$ implies that
	\begin{equation*}
		|\Rm|\leq\tfrac{1}{8}(\tfrac{4}{3})^2\alpha(\tfrac{3}{4}r)^{-2}+\alpha^{-1}r^{-2}
		\leq 2\alpha^{-1}r^{-2}\qquad\text{on }P(x_0,t_0;\tfrac{1}{2}r,-\tfrac{1}{2}r^2),
	\end{equation*}
	in contradiction to our assumption.
\end{proof}

\begin{proof}[Proof of \cref{thm:backward-pseudolocality}]
	The theorem follows by applying \cref{lem:induction-over-scales} to sufficiently small subballs of $B(x_0,t_0,r)$.
\end{proof}

\appendix
\section{Preliminaries}\label{sec:preliminaries}
Here we collect some results established in \cite{he2026pseudolocality} that will be used in the current paper.
Roughly speaking, \cite{he2026pseudolocality} proves pseudolocality for $G_2$-Laplacian flow with bounded torsion and generalizes all the results in \cite{bamler2020entropy} to that case.
Below we write $|E|_{\mathrm{max}}$ for $\sup_{M\times I}|E|$.
Whenever it appears, we are implicitly assuming it is finite, i.e., $\sup_{M\times I}|E|<\infty$.
Similar conventions are made for $R_{\mathrm{min}}$.

By default, $M^n$ is a compact manifold and $(g(t))_{t\in I}$ is a smooth family of Riemannian metrics on it.
\begin{Proposition}
	Let $d\nu=(4\pi\tau)^{-n/2}e^{-f}dV_g$ be the conjugate heat kernel measure based at $(x_0,t_0)\in M\times I$, where $\tau(t)=t_0-t$. Then
	\begin{equation}\label{ineq:energy-estimate}
		\int_M   |\nabla f|^2   d\nu
		\leq\frac{n}{2\tau}-R_{\mathrm{min}}
		+\frac{C(|E|_{\mathrm{max}}\tau\mid n)}{\tau}.
	\end{equation}
\end{Proposition}

\begin{Theorem}[No local collapsing]\label{thm:no-local-collapsing}
	If $R\leq r^{-2}$ on $B(x,t,r)$ and $[t-r^2,t]\subset I$, then
	\begin{equation*}
		|B(x,t,r)|_{g(t)}\geq c(n,R_\mathrm{min}r^2,|E|_{\mathrm{max}}r^2)\exp(\mathcal{N}_{x,t}(r^2))r^n.
	\end{equation*}
\end{Theorem}

\begin{Theorem}[No local inflating]\label{thm:no-local-inflating}
	If $(x,t)\in M\times I$, $\lambda>0$ and $r>0$ satisfies $[t-r^2,t]\subset I$, then
	\begin{equation*}
		|B(x,t,\lambda r)|_{g(t)}
		\leq C(n,R_{\mathrm{min}}r^2,|E|_{\mathrm{max}}r^2,\lambda)\exp(\mathcal{N}_{x,t}(r^2))r^n.
	\end{equation*}
\end{Theorem}

\begin{Theorem}[$\epsilon$-regularity of pointed Nash entropy]\label{thm:refined-epsilon-regularity}
	For any $\epsilon>0$, there exists $\lambda(\epsilon)>0$ such that the following holds.
	
	Let $(M^7,\varphi(t),g(t))_{t\in I}$ be a Laplacian flow over a compact manifold, $(x,t)\in M\times I$ be a point, and $r>0$ be a scale such that $[t-r^2,t]\subset I$.
	Suppose
	\begin{equation*}
		\sup_{M\times I}|T_{\varphi}|^2\leq\lambda r^{-2},\qquad
		\mathcal{N}_{x,t}(r^2)\geq-\lambda.
	\end{equation*}
	Then $|\Rm|\leq\epsilon r^{-2}$ over $P(x,t;\epsilon^{-1}r,-(1-\epsilon)r^2,\epsilon^{-1}r^2)$.
\end{Theorem}

\begin{Theorem}\label{thm:entropy-continuity}
	For any $\epsilon>0$ and $Y<\infty$, there exists $\overline{\lambda}(\epsilon,Y)>0$ such that the following holds for all $0<\lambda<\overline{\lambda}$.
	
	Let $(M^7,\varphi(t),g(t))_{t\in I}$ be a Laplacian flow over a compact manifold, $(x,t)\in M\times I$ be a point, and $r>0$ be a scale such that $[t-r^2,t]\subset I$.
	Suppose
	\begin{equation*}
		\sup_{M\times[t-r^2,t]}|T_{\varphi}|^2\leq\lambda r^{-2},\qquad
		\sup_{P^{-}(x,t;r)}|\Rm|\leq r^{-2},\qquad
		\mathcal{N}_{x,t}(r^2)\geq-Y.
	\end{equation*}
	Then $\mathcal{N}_{x,t}(\lambda r^2)\geq-\epsilon$.
\end{Theorem}

\begin{Theorem}[Pseudolocality]\label{thm:refined-pseudolocality}
	For any $\alpha\in(0,1]$, there is $\delta(\alpha)>0$ so that the following holds.

	Let $(M^7,\varphi(t),g(t))_{t\in[0,T]}$ be a Laplacian flow on a compact manifold satisfying $|T_\varphi|^2\leq\Lambda$ for some $\Lambda>0$.
	Let $x_0\in M$.
	Suppose $T\leq\delta\Lambda^{-1}$ and
	\begin{equation*}
		\inf_{0<t\leq T}
		\mu\big(B(x_0,0,\delta^{-1}\sqrt{t}),g(0),t\big)
		\geq-\delta,
	\end{equation*}
	then
	\begin{equation*}
		|\Rm|(x,t)\leq\frac{\alpha}{t},
		\quad\forall\,x\in B(x_0,t,\alpha^{-1}\sqrt{t}),\ t\in(0,T].
	\end{equation*}
\end{Theorem}

\section{Backward pseudolocality of Ricci flow}\label{sec:Ricci-flow}
Applying the proof in \S\ref{sec:estimate-on-Nash-entropy}, \S\ref{sec:proof-of-main-theorem} directly to Ricci flow only yields backward pseudolocality under an additional assumption on scalar curvature upper bound.
In this section, we give the necessary modifications needed for proving Bamler's version of backward pseudolocality (\cref{thm:backward-pseudolocality-of-Ricci-flow}).
Besides, many estimates in \S\ref{sec:estimate-on-Nash-entropy} can be reformulated in a cleaner way in the case of Ricci flow.

Let us first recall the general backward pseudolocality theorem for Ricci flow.
\begin{Theorem}[\cite{bamler2020structure} Theorem 2.47]\label{thm:backward-pseudolocality-of-Ricci-flow}
	For any $n\in\N$ and $\alpha>0$, there is an $\epsilon(n,\alpha)>0$ such that the following holds.

	Let $(M^n,g(t))_{t\in I}$ be a Ricci flow on a compact manifold, $(x_0,t_0)\in M\times I$ be a point, and $r>0$ be a scale such that $[t_0-r^2,t_0]\subset I$.
	Suppose
	\begin{equation*}
		|B(x_0,t_0,r)|\geq\alpha r^n,\qquad
		|\Rm|\leq(\alpha r)^{-2}\text{ over }B(x_0,t_0,r).
	\end{equation*}
	Then $|\Rm|\leq(\epsilon r)^{-2}$ over $P(x_0,t_0;(1-\alpha)r,-(\epsilon r)^2)$.
\end{Theorem}

Our proof still needs Bamler's induction over scales argument \cite[\S 26]{bamler2020structure}, which reduces the problem to showing:
\begin{Lemma}[\cite{bamler2020structure} Lemma 26.1]\label{lem:Bamler-Lemma-26-1}
	For every $\alpha>0$, there is an $\epsilon(n,\alpha)>0$ such that the following holds.

	Let $(M^n,g(t))_{t\in I}$ be a Ricci flow on a compact manifold, $(x_0,t_0)\in M\times I$ be a point, and $r>0$ be a scale such that $[t_0-r^2,t_0]\subset I$.
	Suppose
	\begin{enumerate}
		\item \label{lem-it:non-collapsing} $|B(x_0,t_0,r)|\geq\alpha r^n$;
		\item \label{lem-it:curvature-bound} $|\Rm|\leq\alpha^{-1}r^{-2}$ on $P(x_0,t_0;r,-(\alpha r)^2)$;
		\item \label{lem-it:scalar-curvature-upper-bound} $R(x_0,t_0)\leq\epsilon r^{-2}$;
		\item \label{lem-it:scalar-curvature-lower-bound} $R(\cdot,t_0-r^2)\geq-\epsilon r^{-2}$.
	\end{enumerate}
	Then $|\Rm|\leq\alpha r^{-2}+\sup_{B(x_0,t_0,r)}|\Rm|$ and $|\Ric|\leq\alpha r^{-2}$ on $P(x_0,t_0;(1-\alpha)r,-(1-\alpha)r^2)$.
\end{Lemma}
Note the scalar curvature upper bound is only assumed at $(x_0,t_0)$. (By maximum principle the lower bound holds on the whole spacetime in consideration.)
So our previous \cref{lem:almost-static}(\ref{lem-it:almost-static-curvature-bound}) breaks down in this case.

We first give a cleaner reformulation of the results in \S\ref{sec:estimate-on-Nash-entropy} for Ricci flow.
Let $(M^n,g(a))_{a\in I}$ be a Ricci flow on a compact manifold and $s<t$ be such that $[s,t]\subset I$.
Denote $\tau:=t-s$.
Recall \cref{eq:Ricci-curvature-integral-in-the-case-of-Ricci-flow} asserts $2\mathcal I(x,t;\tau)=R(x,t)-P_{s,t}R_s(x)$.

\begin{itemize}
	\item (\cref{prop:commutator}). For any $\phi\in C^\infty(M)$, we have
	\begin{equation*}
		|\Delta_tP_{s,t}\phi-P_{s,t}\Delta_s\phi|^2
		\leq\big(  R(x,t)-P_{s,t}R_s(x)  \big)
		\left(  P_{s,t}|\nabla\phi|_s^2(x)-|\nabla P_{s,t}\phi|_t^2(x)  \right).
	\end{equation*}
	Here we use \cref{eq:Ricci-curvature-integral-in-the-case-of-Ricci-flow} and the fact \cite[(3.8)]{hein2014new}
	\begin{equation*}
		2\int_s^t  P_{a,t}|\nabla^2P_{s,a}\phi|_a^2(x)  \,da
		=P_{s,t}|\nabla\phi|_s^2(x)-|\nabla P_{s,t}\phi|_t^2(x).
	\end{equation*}
	\item (\cref{prop:derivative-of-Nash-entropy-in-time-variable}). Denote $\overline R:=P_{s,t}R_s(x)$ for simplicity.
	Then
	\begin{equation*}
		\left|\partial_t\mathcal N_{x,t}(\tau)+\overline R\right|
		\leq\sqrt{\big(R(x,t)-\overline R\big)\left(\frac{n}{2\tau}-\overline R\right)}.
	\end{equation*}
	In particular, if $R\geq-\lambda$, then
	\begin{equation}\label{ineq:bound-on-Nash-entropy-for-Ricci-flow}
		\partial_t\mathcal N_{x,t}(\tau)
		\leq\lambda+\sqrt{\big(R(x,t)+\lambda\big)\left(\frac{n}{2\tau}+\lambda\right)}.
	\end{equation}
	\item (\cref{lem:small-Ricci-curvature-integral} and \cref{lem:almost-static}(\ref{lem-it:existence-of-cut-off})). In this case, no cut-off is needed to obtain the weighted Ricci energy bound.
	It is replaced by the identity \cref{eq:Ricci-curvature-integral-in-the-case-of-Ricci-flow}.
	More precisely, \cref{lem:Bamler-Lemma-26-1}(\ref{lem-it:scalar-curvature-upper-bound})(\ref{lem-it:scalar-curvature-lower-bound})  and \cref{eq:Ricci-curvature-integral-in-the-case-of-Ricci-flow} imply
	\begin{equation}\label{ineq:smallness-of-Ricci-curvature-integral-for-Ricci-flow}
		\mathcal I(x_0,t_0;r^2)=\tfrac{1}{2}\big(R(x_0,t_0)-P_{t_0-r^2,t_0}R_{t_0-r^2}(x_0)\big)\leq\epsilon r^{-2}.
	\end{equation}
	This suffices for our purpose.
	\item (\cref{lem:almost-static}(\ref{lem-it:almost-static-curvature-bound})). As mentioned before, this cannot be applied directly in the absence of a scalar curvature upper bound.
\end{itemize}

Now we show how to use \cref{ineq:bound-on-Nash-entropy-for-Ricci-flow}, \cref{ineq:smallness-of-Ricci-curvature-integral-for-Ricci-flow}, Perelman's heat kernel lower bound, and the results in \cite{bamler2020entropy} to propagate smallness of pointed Nash entropy backward as in the proof of \cref{lem:extension}.
\begin{proof}[Schematic of proof of \cref{lem:Bamler-Lemma-26-1}]
	We may assume $t_0=0$ and $r=1$.
	Set $A:=\overline{B(x_0,0,1-\tfrac{\alpha}{2})}$ and fix $t_c\in(-\alpha^2/4,0)$.
	By (\ref{lem-it:non-collapsing})(\ref{lem-it:curvature-bound}) and \cite[Theorem 10.4]{bamler2020entropy}, there is $\rho(n,\alpha)>0$ such that
	\begin{equation*}
		\mathcal N_{y,t_c}(\rho^2)\geq-\eta/4
	\end{equation*}
	for all $y\in A$.
	Here $\eta(n)>0$ denotes the dimensional constant in the epsilon-regularity theorem \cite[Theorem 10.2]{bamler2020entropy}.
	Let
	\begin{equation*}
		t_*:=\inf\left\{t\in[-(1-\alpha),t_c]:\mathcal N_{y,s}(\rho^2)\geq-\eta\text{ for all }(y,s)\in A\times[t,t_c]\right\}.
	\end{equation*}
	The epsilon-regularity supplies curvature bound and Perelman's heat kernel lower bound \cite[Corollary 9.5]{perelman2002entropy} then gives (for example, by considering a curve traveling from $(x_0,0)$ to $(y,s)$ first and then to $(z,u)$)
	\begin{equation*}
		K(x_0,0;z,u)\geq c(n,\alpha,t_c)>0,\qquad\text{for all }(z,u)\in B(y,s,\delta)\times[s-2\delta^2,s-\delta^2]
	\end{equation*}
	for some $\overline\delta(n,\alpha)>0$ and all $\delta\in(0,\overline{\delta})$.
	The curvature bound and assumption (\ref{lem-it:non-collapsing}) also give a uniform lower bound on volume ratio.
	So it follows from \cref{ineq:smallness-of-Ricci-curvature-integral-for-Ricci-flow} and the bounds on $\nabla\Ric$, $\partial_t\Ric$ (obtained from Shi-type estimates) that
	\begin{equation}\label{ineq:smallness-of-Ricci}
		\sup_{A\times[t_*,t_c]}|\Ric|\leq C(\epsilon\mid n,\alpha,t_c).
	\end{equation}
	This provides the scalar curvature upper bound needed for the application of \cref{ineq:bound-on-Nash-entropy-for-Ricci-flow}, which implies
	\begin{equation*}
		\partial_s\mathcal N_{y,s}(\rho^2)\leq C(\epsilon\mid n,\alpha,t_c)
	\end{equation*}
	for all $(y,s)\in A\times[t_*,t_c]$.
	Taking $\epsilon$ sufficiently small forces $t_*=-(1-\alpha)$.
	Interpolating between \cref{ineq:smallness-of-Ricci} and the higher order derivative bounds on curvature (\cref{lem:interpolation}) gives
	\begin{equation*}
		\left|\frac{d}{dt}|\Rm|\right|
		\leq C(n)\big(|\nabla^2\Ric|+|\Rm|\,|\Ric|\big)
		\leq C(\epsilon\mid n,\alpha,t_c)
	\end{equation*}
	over $A\times[-(1-\alpha),t_c]$.
	Clearly, $|\Rm(y,t_c)-\Rm(y,0)|\leq C(n,\alpha)|t_c|$ and similar for $\Ric$.
	The final results follow by taking $t_c$ sufficiently close to $0$ first and then $\epsilon$ sufficiently small.
\end{proof}

\newpage
\bibliography{backward_pseudolocality}
\end{document}